\documentclass{amsart}
\usepackage{amssymb,amscd}
\usepackage[cp1251]{inputenc}
\usepackage[T2A]{fontenc}

\newtheorem{theorem}{Theorem}[section]
\newtheorem{proposition}[theorem]{Proposition}
\newtheorem{corollary}[theorem]{Corollary}
\newtheorem{lemma}[theorem]{Lemma}
\newtheorem{question}[theorem]{Question}
\theoremstyle{definition}

\newtheorem{example}[theorem]{Example}
\newtheorem{construction}[theorem]{Construction}
\theoremstyle{remark}
\newtheorem*{remark}{Remark}
\newtheorem*{acknowledgments}{Acknowledgements}

\numberwithin{equation}{section}
\numberwithin{figure}{section}

\def\C{\mathbb C}
\def\D{\mathbb D}
\def\I{\mathbb I}

\def\R{\mathbb R}
\def\T{\mathbb T}
\def\Z{\mathbb Z}

\def\sK{\mathcal K}

\def\phi{\varphi}

\newcommand{\mb}[1]{{\textbf {\textit#1}}}

\newcommand{\sbr}[2]{{\textstyle\genfrac{[}{]}{}{}{#1}{#2}}}

\renewcommand{\ge}{\geqslant}
\renewcommand{\le}{\leqslant}

\renewcommand{\Re}{\mathop{\mathrm{Re}}}

\newcommand{\Ker}{\mathop{\rm Ker}}

\newcommand{\cc}{\mathop{\rm cc}}
\def\cone{\mathop{\mathrm{cone}}}
\newcommand{\lk}{\mathop{\rm lk}\nolimits}
\newcommand{\st}{\mathop{\rm St}\nolimits}

\newcommand{\zk}{\mathcal Z_{\mathcal K}}
\newcommand{\zp}{\mathcal Z_P}

\begin{document}

\title{Complex-analytic structures on moment-angle manifolds}

\author{Taras Panov}
\address{Department of Mathematics and Mechanics, Moscow
State University, Leninskie Gory, 119991 Moscow, Russia,
\newline\indent Institute for Theoretical and Experimental Physics,
Moscow, Russia,\quad \emph{and}
\newline\indent Institute for Information Transmission Problems,
Russian Academy of Sciences} \email{tpanov@mech.math.msu.su}

\author{Yuri Ustinovsky}
\address{Department of Mathematics and Mechanics, Moscow
State University, Leninskie Gory, 119991 Moscow, Russia}
\email{yura.ust@gmail.com}

\thanks{The first author was
supported by the Russian Foundation for Basic Research,
grants~11-01-90413-Укр and~11-01-00694, and a grant from Dmitri
Zimin's `Dynasty' foundation. Both authors were supported by
 grants НШ-4995-2012.1 and МД-2253.2011.1 from the President of
Russia.}



\begin{abstract}
We show that the moment-angle manifolds corresponding to complete
simplicial fans admit non-K\"ahler complex-analytic structures.
This generalises the known construction of complex-analytic
structures on polytopal moment-angle manifolds, coming from
identifying them as LVM-manifolds. We proceed by describing
Dolbeault cohomology and some Hodge numbers of moment-angle
manifolds by applying the Borel spectral sequence to holomorphic
principal bundles over toric varieties.
\end{abstract}

\noindent\emph{To Sabir Gusein-Zade, on the occasion of his 60th birthday}

\medskip

\maketitle

\section{Introduction}
\emph{Moment-angle complexes} $\zk$ are spaces acted on by a torus
and parametrised by finite simplicial complexes~$\sK$. They are
central objects in toric topology, and currently are gaining much
interest in the homotopy theory. Due the their combinatorial
origins, moment-angle complexes also find applications in
combinatorial geometry and commutative algebra.

The construction of the moment-angle complex $\zk$ ascends to the
work of Davis--Januszkiewicz~\cite{da-ja91} on (quasi)toric
manifolds; later $\zk$ was described in~\cite{bu-pa99} as a
certain complex build up from polydiscs and tori. In the case when
$\sK$ is a triangulation (simplicial subdivision) of a sphere,
$\zk$ is a topological manifold, referred to as the
\emph{moment-angle manifold}. Dual triangulations of simple convex
polytopes~$P$ provide an important subclass of sphere
triangulations; the corresponding \emph{polytopal} moment-angle
manifolds are known to be smooth~\cite[Lemma~6.2]{bu-pa02} and
denoted~$\zp$. The manifolds $\zp$ corresponding to \emph{Delzant
polytopes} $P$ are closely related to the construction of
\emph{Hamiltonian toric manifolds} via symplectic reduction; $\zp$
arises as the level set for an appropriate moment map and
therefore embeds into $\C^m$ as a nondegenerate intersection of
real quadrics with rational coefficients~\cite{bu-pa02}. The
topology of moment-angle complexes and manifolds is quite
complicated even for small $\sK$ and $P$. The cohomology ring of
$\zk$ was described in~\cite[Th.~4.2,~4.6]{bu-pa99} (see
also~\cite[\S4]{pano08}), and explicit homotopy and diffeomorphism
types for certain particular families of $\zk$ and $\zp$ were
described in~\cite{gr-th07} and~\cite{gi-lo09} respectively.

On the other hand, manifolds obtained as intersections of quadrics
appeared in holomorphic dynamics as the spaces of leaves for
holomorphic foliations in~$\C^m$. Their topology was studied
in~\cite{lope89}; this study led to a discovery of a new class of
compact non-K\"ahler complex-analytic manifolds in the work of
Lopez de Medrano and Verjovsky~\cite{lo-ve97} and
Meersseman~\cite{meer00}, now known as the \emph{LVM-manifolds}.
Bosio and Meersseman were first to observe that the smooth
manifolds underlying a large class of LVM-manifolds are exactly
polytopal moment-angle manifolds, cf.~\cite{bo-me06}. It therefore
became clear that the moment-angle manifolds $\zp$ admit
non-K\"ahler complex-analytic (or shortly complex) structures
generalising the known families of Hopf and Calabi--Eckmann
manifolds.

The aim of this paper is twofold. First, we intended to give an
explicit intrinsic construction of complex structures on manifolds
$\zp$ within the theory of moment-angle complexes, without
referring to the much developed theory of LVM-manifolds (although
of course using ideas and methodology of this theory). Second, we
aimed at generalising the construction of complex structures on
$\zp$ to the nonpolytopal case (this question was also raised
in~\cite[\S15]{bo-me06}), and extending our knowledge of
invariants of these structures, such as Dolbeault cohomology and
Hodge numbers.

In Section~2 we review the construction of moment-angle complexes
and related spaces, such as the complement $U(\sK)$ of the
coordinate subspace arrangement in $\C^m$ corresponding to~$\sK$,
where $m$ is the number of vertices of~$\sK$. We then restrict
attention to the case when $\sK$ is the underlying complex of a
complete simplicial (but not necessarily rational) fan $\Sigma$
in~$\R^n$. Such $\sK$ are also known as \emph{starshaped spheres}.
We show that the corresponding moment-angle manifold $\zk$ admits
a smooth structure as the transverse space to the orbits of an
action of $\R^{m-n}$ on~$U(\sK)$. The situation here is similar to
the polytopal case of~$\zp$, however, we do not have an explicit
presentation of~$\zk$ as an intersection of quadrics. This also
leaves open a question of whether moment-angle complexes $\zk$
corresponding to more general sphere triangulations~$\sK$ can be
smoothed.

In Section~3 we modify the construction of Section~2 by replacing
the action of $\R^{m-n}$ on $U(\sK)$ by a holomorphic action of a
complex group $C$ isomorphic to~$\C^\ell$ where $m-n=2\ell$,
provided that $m-n$ is even (this can always be achieved by adding
a `ghost' vertex to~$\sK$). The identification of $\zk$ with the
quotient $U(\sK)/C$ endows it with a structure of a complex
manifold of dimension~$m-\ell$.\footnote{While preparing this
paper for publication we discovered that results similar to those
of Section~3 were obtained by Tambour~\cite{tamb10} using a
different approach. In particular, Tambour constructed complex
structures on manifolds $\zk$ coming from \emph{rationally}
starshaped spheres~$\sK$, by relating them to a class of
generalised LVM-manifolds described by Bosio in~\cite{bosi01}.}

In Section~4 we restrict to the polytopal case and relate the
complex structure on $\zp$ coming from the normal fan of $P$ via
the construction of Section~3 to the complex structure coming from
identifying $\zp$ as an LVM-manifold~\cite{bo-me06}. We also
relate the quadratic description of $\zp$ to the quadratic
description of LVM-manifolds.

In Section~5 we consider rational fans $\Sigma$, which give rise
to toric varieties~$X_\Sigma$. The Cox construction identifies
$X_\Sigma$ with the quotient of $U(\sK)$ by an action of an
algebraic group $G$ of dimension~$m-n=2\ell$. In the case of
nonsingular $X_\Sigma$ the inclusion of the complex
$\ell$-dimensional group $C$ into the algebraic group $G$ gives
rise to a holomorphic principal bundle $\zk\to X_\Sigma$ whose
fibre is a compact complex torus of dimension~$\ell$. This extends
the construction of~\cite{me-ve04} of holomorphic principal
bundles over projective toric varieties to the nonpolytopal (and
therefore non projective) case.\footnote{In view of the results
of~\cite{tamb10}, this extension is covered by the results of
Cupit-Foutou and Zaffran~\cite{cu-za07}.} An application of the
Borel spectral sequence to the holomorphic bundle $\zk\to
X_\Sigma$ allows us to describe the Dolbeault cohomology groups
of~$\zk$ in Theorem~\ref{dolbzp}. From this description some Hodge
numbers $h^{p,q}(\zk)$ may be calculated explicitly
(Theorem~\ref{hodge}).

\begin{acknowledgments}
The authors are grateful to an anonymous referee who pointed out
an inaccuracy in the original argument for the properness of the
action in Theorem~\ref{zksmooth}, and suggested a way to correct
it. We also thank another referee for pointing out a connection
between the results of~\cite{tamb10} and~\cite{cu-za07}, which
affected our exposition, and for other useful comments and
suggestions.
\end{acknowledgments}

\section{Moment-angle complexes and manifolds}
Let $\mathcal K$ be an abstract simplicial complex on the set
$[m]=\{1,\ldots,m\}$, i.e. a collection of subsets
$I=\{i_1,\ldots,i_k\}\subset[m]$ closed under inclusion. We refer
to $I\in\sK$ as \emph{simplices} and always assume that
$\varnothing\in\sK$. We denote by~$|\sK|$ a geometric realisation
of~$\sK$, which is a topological space.

Consider the unit polydisc in $\C^m$,
\[
  \D^m=\bigl\{ (z_1,\ldots,z_m)\in\C^m\colon |z_i|\le1,\quad i=1,\ldots,m
  \bigr\}.
\]
Given $I\subset[m]$, define
\[
  B_I:=\bigl\{(z_1,\ldots,z_m)\in
  \D^m\colon |z_j|=1\text{ for }j\notin I\bigl\},
\]
Following~\cite{bu-pa02}, define the \emph{moment-angle complex}
$\zk$ as
\begin{equation}\label{defzk}
  \zk:=\bigcup_{I\in\sK}B_I\subset\D^m
\end{equation}
It is invariant under the coordinatewise action of the standard
torus
\[
  \T^m=\bigl\{(z_1,\ldots,z_m)\in\C^m\colon |z_i|=1,\quad i=1,\ldots,m
  \bigr\}.
\]

The definition of the moment-angle complex $\zk$ is a particular
case of the following general construction.

\begin{construction}[$\sK$-power]
Let $X$ be a space, and $W$ a subspace of $X$. Given
$I\subset[m]$, set
\begin{equation}\label{xwi}
  (X,W)^I:=\bigl\{(x_1,\ldots,x_m)\in X^m\colon x_j\in A\text{ for
  }j\notin I\bigr\}\cong\prod_{i\in I}X\times\prod_{i\notin I}W,
\end{equation}
and define the \emph{$\sK$-power} (also known as the
\emph{polyhedral product}) of $(X,W)$ as
\[
  (X,W)^\sK:=\bigcup_{I\in\sK}(X,W)^I\subset X^m.
\]
Note that we do not assume that $\sK$ contains all one-element
subsets $\{i\}\subset[m]$; we refer to $\{i\}\notin\sK$ as a
\emph{ghost vertex}. Obviously, if $\sK$ has $k$ ghost vertices,
then $(X,W)^\sK\cong(X,W)^{\mathcal K'}\times W^k$, where
$\mathcal K'$ does not have ghost vertices.

It follows from the definition that
\begin{itemize}
\item$\zk=(\D,\T)^\sK$, where $\T$ is the unit circle.
\end{itemize}
Other important particular cases of $\sK$-powers include;
\begin{itemize}
\item
the standard cubical decomposition of the quotient
$\zk/\T^m\cong\cone|\sK|$ (see~\cite[\S6.2]{bu-pa02}), namely,
\begin{equation}\label{cck}
  \cc(\sK):=(\I,1)^\sK\subset\I^m,
\end{equation}
where $\I$ is the unit segment $[0,1]$ and $\I^m$ is the unit
$m$-cube;
\item the \emph{complex coordinate subspace arrangement complement} corresponding to $\sK$
(see~\cite[\S8.2]{bu-pa02}):
\[
  U(\sK):=\C^m\setminus\bigcup_{\{i_1,\ldots,i_k\}\notin\sK}
  \{\mb z\in\C^m\colon z_{i_1}=\ldots=z_{i_k}=0\},
\]
namely,
\[
  U(\sK)=(\C,\C^\times)^\sK,
\]
where $\C^\times=\C\setminus\{0\}$.
\end{itemize}
We obviously have $\zk\subset U(\sK)$. Moreover, $\zk$ is a
$\T^m$-equivariant deformation retract of $U(\sK)$ for every
$\sK$~\cite[Th.~8.9]{bu-pa02}. For more generalisations and
applications of $\sK$-powers see~\cite{b-b-c-g10}.
\end{construction}

It is shown in~\cite[Lemma~6.13]{bu-pa02} that $\zk$ is a (closed)
topological manifold whenever $|\sK|$ is a triangulation of a
sphere; in this case we refer to $\zk$ as a \emph{moment-angle
manifold}. If $\sK=\sK_P$ is the dual triangulation of a simple
convex polytope~$P$, then $\zp:=\mathcal Z_{\sK_P}$ can be
canonically smoothed~\cite[Lemma~6.2]{bu-pa02} using the standard
structure of manifold with corners on~$P$.

The aim of this section is to show that moment-angle manifolds
$\zk$ admit smooth structures for a wider class of simplicial
complexes $\sK$, namely, those corresponding to \emph{complete
simplicial fans}.

A set of vectors $\mb a_1,\ldots,\mb a_k\in\R^n$ defines a convex
polyhedral \emph{cone}
\[
  \sigma=\{\mu_1\mb a_1+\ldots+\mu_k\mb
  a_k\colon\mu_i\in\R,\;\mu_i\ge0\}.
\]
A cone is \emph{rational} if its generating vectors can be chosen
from the integer lattice $\Z^n\subset\R^n$, and is \emph{strongly
convex} if it does not contain a line. A cone is \emph{simplicial}
(respectively, \emph{regular}) if it is generated by a part of
basis of $\R^n$ (respectively, $\Z^n$).

A \emph{fan} is a finite collection
$\Sigma=\{\sigma_1,\ldots,\sigma_s\}$ of strongly convex cones in
some $\R^n$ such that every face of a cone in $\Sigma$ belongs to
$\Sigma$ and the intersection of any two cones in $\Sigma$ is a
face of each. A fan $\Sigma$ is \emph{rational} (respectively,
\emph{simplicial}, \emph{regular}) if every cone in $\Sigma$ is
rational (respectively, simplicial, regular). A fan
$\Sigma=\{\sigma_1,\ldots,\sigma_s\}$ is called \emph{complete} if
$\sigma_1\cup\ldots\cup\sigma_s=\R^n$.

In this section we assume that $\Sigma$ is a simplicial (but not
necessarily rational) fan in $\R^n$ with $m$ one-dimensional
cones, for which we choose generator vectors $\mb a_1,\ldots,\mb
a_m$. Define the \emph{underlying simplicial complex} $\sK_\Sigma$
on $[m]$ as the collection of subsets $I\subset[m]$ such that
$\{\mb a_i\colon i\in I\}$ spans a cone of~$\Sigma$. Note that
$\Sigma$ is complete if and only if $|\sK_\Sigma|$ is a
triangulation of~$S^{n-1}$. Now consider the linear map
\begin{equation}\label{lambdar}
  \Lambda_\R\colon\R^m\to\R^n,\quad\mb e_i\mapsto\mb a_i,
\end{equation}
where $\mb e_1,\ldots,\mb e_m$ is the standard basis of~$\R^m$.
Let
\[
  \R^m_>=\{(y_1,\ldots,y_m)\in\R^m\colon y_i>0\}
\]
be the multiplicative group of $m$-tuples of positive real
numbers, and define
\begin{equation}\label{rsigma}
  R_\Sigma:=\exp(\Ker\Lambda_\R)=
  \bigl\{(y_1,\ldots,y_m)\in\R^m_>\colon
  \prod_{i=1}^my_i^{\langle\mb a_i,\mb u\rangle}=1\text{ for all }\mb u\in\R^n\bigr\},
\end{equation}
where $\langle\:,\;\rangle$ is the standard scalar product
in~$\R^n$. Note that $R_\Sigma\cong\R^{m-n}_>$ if $\Sigma$ is
complete (or contains at least one $n$-dimensional cone).

We let $\R^m_>$ act on the complement $U(\sK_\Sigma)$ by
coordinatewise multiplications and consider the restricted action
of the subgroup $R_\Sigma\subset\R^m_>$. Recall that an action of
a topological group $G$ on a space $X$ is \emph{proper} if the map
$h\colon G\times X\to X\times X$, \ $(g,x)\mapsto (gx,x)$ is
proper.

\begin{theorem}\label{zksmooth}
Let $\Sigma$ be a complete simplicial fan in~$\R^n$ with $m$
one-dimensional cones, and let $\sK=\sK_\Sigma$ be its underlying
simplicial complex. Then
\begin{itemize}
\item[(a)]
the group $R_\Sigma$ acts on $U(\sK)$ freely and properly, and the
quotient $U(\sK)/R_\Sigma$ is a smooth $(m+n)$-dimensional
manifold;

\item[(b)] $U(\sK)/R_\Sigma$ is $\T^m$-equivariantly
homeomorphic to~$\zk$.
\end{itemize}
Therefore, $\zk$ can be smoothed.
\end{theorem}
\begin{proof}
We first prove statement~(a). The fact that $R_\Sigma$ acts on
$U(\sK)$ freely is standard, and is proved in the same way
as~\cite[Prop.~6.5]{bu-pa02}. Indeed, a point $\mb z\in U(\sK)$
has a nontrivial isotropy subgroup with respect to the action of
$\R^m_>$ only if some of its coordinates vanish. These
$\R^m_>$-isotropy subgroups are of the form $(\R_>,1)^I$,
see~\eqref{xwi}, for some $I\in\sK$. The restriction of
$\exp\Lambda_\R$ to every such $(\R_>,1)^I$ is an injection.
Therefore, $R_\Sigma=\exp(\Ker\Lambda_\R)$ intersects every
$\R^m_>$-isotropy subgroup only at the unit, which implies that
the $R_\Sigma$-action on $U(\sK)$ is free.

Let us prove that the $R_\Sigma$-action on~$U(\sK)$ is proper. Let
$\{\mb g^i\}\in R_\Sigma$, $\{\mb x^i\}\in U(\sK)$ be sequences of
points such that $\{\mb x^i\}$ and $\{\mb y^i\}:=\{\mb g^i\mb
x^i\}$ have limits in $U(\sK)$:
\[
  \{\mb x^i\}\to {\mb x}=(x_1,\dots,x_m),\quad
  \{\mb y^i\}\to {\mb y}=(y_1,\dots, y_m).
\]
We claim that a subsequence of $\{\mb g^i\}$ has limit
in~$R_\Sigma$. Indeed, every $\mb g^i$ is represented by an
$m$-tuple,
\[
\mb g^i=(\exp{\gamma^i_1},\ldots,\exp{\gamma^i_m})\in\exp\R^m.
\]
Passing to a subsequence of $\{\mb g^i\}$ if necessary, we may
assume that every sequence $\{\gamma^i_k\}$, \ $k=1,\ldots,m$, has
a finite or infinite limit (including $\pm\infty$). Let
\[
  I_+=\{k\colon\gamma^i_k\to +\infty\}\subset[m],\quad
  I_-=\{k\colon\gamma^i_k\to -\infty\}\subset[m].
\]
Since sequences $\{\mb x^i\}$ and $\{\mb y^i\}$ are bounded,
$x_k=0$ for $k\in I_+$ and $y_k=0$ for $k\in I_-$. The definition
of $U(\sK)$ implies that $I_+$ and $I_-$ are simplices of $\sK$.
Let $\sigma_+,\sigma_-$ be the corresponding cones of the
fan~$\Sigma$. Since $\sigma_+\cap\sigma_-=\{0\}$, there exists a
linear function $\xi$ on $\R^n$ such that $\xi(v)>0$ for $v\in
\sigma_+, v\neq 0$, and $\xi(v)<0$ for $v\in \sigma_-, v\neq 0$.
Recall that $\mb g^i\in R_\Sigma=\exp(\Ker\Lambda_\R)$, therefore,
\[
  0=\xi\biggl(\sum_{k=1}^m \gamma^i_k \mb a_k\biggr)=\sum_{k=1}^m \gamma^i_k \xi(\mb
  a_k).
\]
Hence, both $I_+$ and $I_-$ are empty (otherwise the latter sum
tends to infinity). Thus, $\mb g^i$ converges to a point in
$R_\Sigma$, and the preimage of any compact subspace in
$U(\sK)\times U(\sK)$ under the action map $h\colon R_\Sigma\times
U(\sK)\to U(\sK)\times U(\sK)$ is compact.
This proves the properness of the action. Since the Lie group
$R(\Sigma)$ acts smoothly, freely and properly on
$U(\sK)$, the orbit space
$U(\sK)/R(\Sigma)$ admits a structure of a smooth manifold by the
standard result~\cite[Th.~9.16]{lee00}.

In our case it is possible to construct a smooth atlas on
$U(\sK)/R_\Sigma$ explicitly. To do this, it is convenient to
pre-factorise everything by the action of $\T^m$. The quotient
$U(\sK)/\T^m$ has the following decomposition as a $\sK$-power:
\[
  U(\sK)/\T^m=(\R_\ge,\R_>)^\sK=\bigcup_{I\in\sK}(\R_\ge,\R_>)^I,
\]
where $\R_\ge$ is the set of nonnegative reals. Since the fan
$\Sigma$ is complete, we may take the union above only over
$n$-element simplices $I=\{i_1,\ldots,i_n\}\in\sK$. Consider one
such simplex~$I$; the generators of the corresponding
$n$-dimensional cone $\sigma\in\Sigma$ are $\mb a_{i_1},\ldots,\mb
a_{i_n}$. Let $\mb u_1,\ldots,\mb u_n$ denote the dual basis of
$\R^n$ (which is a generator set of the dual cone $\sigma^*$).
Then we have $\langle\mb a_{i_k},\mb u_j\rangle=\delta_{jk}$. Now
consider the map
\begin{align*}
  p_I\colon(\R_\ge,\R_>)^I&\to\R^n_\ge\\
  (y_1,\ldots,y_m)&\mapsto
  \Bigr(\prod_{i=1}^my_i^{\langle\mb a_i,\mb u_1\rangle},\ldots,\:
  \prod_{i=1}^my_i^{\langle\mb a_i,\mb u_n\rangle}\Bigl),
\end{align*}
where we set $0^0=1$. Note that zero cannot occur with a negative
exponent in the right hand side, hence $p_I$ is well defined as a
continuous map. Every $(\R_\ge,\R_>)^I$ is $R_\Sigma$-invariant,
and it follows from~\eqref{rsigma} that $p_I$ induces an injective
map
\[
  q_I\colon(\R_\ge,\R_>)^I/R_\Sigma\to\R^n_\ge.
\]
This map is also surjective since every
$(x_1,\ldots,x_n)\in\R^n_\ge$ is covered by $(y_1,\ldots,y_m)$
where $y_{i_j}=x_j$ for $1\le j\le n$ and $y_k=1$ for
$k\notin\{i_1,\ldots,i_n\}$. Hence, $q_I$ is a homeomorphism. It
is covered by a $\T^m$-equivariant homeomorphism
\[
  \overline q_I\colon
  (\C,\C^\times)^I/R_\Sigma\to\C^n\times\T^{m-n},
\]
where $\C^n$ is identified with a quotient of $\R_\ge^n\times\T^n$
in the standard way (e.g., using the polar coordinates in each
factor). Since $U(\sK)/R_\Sigma$ is covered by open subsets
$(\C,\C^\times)^I/R_\Sigma$, and $\C^n\times\T^{m-n}$ embeds as an
open subset in $\R^{m+n}$, the set of homeomorphisms $\{\overline
q_I\colon I\in\sK\}$ provides an atlas for $U(\sK)/R_\Sigma$
(compare~\cite[proof of Lemma~6.2]{bu-pa02}). The change of
coordinates transformations $\overline q_J\overline
q_I^{-1}\colon\C^n\times\T^{m-n}\to\C^n\times\T^{m-n}$ are smooth
by inspection; thus $U(\sK)/R_\Sigma$ is a smooth manifold.

\begin{remark}
The set of homeomorphisms
$\{q_I\colon(\R_\ge,\R_>)^I/R_\Sigma\to\R^n_\ge\}$ defines a
canonical atlas for the smooth manifold with corners $\zk/\T^m$.
If $\sK=\sK_P$ for a simple polytope $P$, then this smooth
structure with corners coincides with that of~$P$.
\end{remark}

If $X$ is a Hausdorff locally compact space with a proper
$G$-action, and $Y\subset X$ a compact subspace which intersects
every $G$-orbit at a single point, then $Y$ is homeomorphic to the
orbit space $X/G$. Therefore, to prove statement~(b) it is enough
to verify that every
$R_\Sigma$-orbit intersects $\zk\subset U(\sK)$ at a single point. 
We first prove that the $R_\Sigma$-orbit of any $\mb y\in
U(\sK)/\T^m=(\R_\ge,\R_>)^\sK$ intersects $\zk/\T^m$ at a single
point. For this we use cubical decomposition $\cc(\sK)$ of
$\zk/\T^m$, see~\eqref{cck}.

Assume first that $\mb y\in\R^m_>$. The $R_\Sigma$-action on
$\R^m_>$ is obtained by exponentiating the linear action of
$\Ker\Lambda_\R$ on $\R^m$. Consider the subset
$(\R_\le,0)^\sK\subset\R^m$, where $\R_\le$ denotes the set of
nonpositive reals. It is taken by the exponential map
$\exp\colon\R^m\to\R^m_>$ homeomorphically onto
$((0,1],1)^\sK=\cc^\circ(\sK)\subset\R^m_>$, where the latter is
the relative interior of $\cc(\sK)$. The map
\begin{equation}\label{oto}
  \Lambda_\R\colon(\R_\le,0)^\sK\to\R^n
\end{equation}
takes every $(\R_\le,0)^I$ to $-\sigma$, where $\sigma\in\Sigma$
is the cone corresponding to $I\in\sK$. Since $\Sigma$ is
complete, map~\eqref{oto} is one-to-one.

The orbit of $\mb y$ under the action of $R_\Sigma$ coincides with
the set of points $\mb w\in\R^m_>$ such that $\exp\Lambda_\R\mb
w=\exp\Lambda_\R\mb y$. Since $\Lambda_\R\mb y\in\R^n$ and
map~\eqref{oto} is one-to-one, there is a unique point $\mb
y'\in(\R_\le,0)^\sK$ such that $\Lambda_\R\mb y'=\Lambda_\R\mb y$.
Since $\exp\Lambda_\R\mb y'\subset\cc^\circ(\sK)$, the
$R_\Sigma$-orbit of $\mb y$ intersects $\cc^\circ(\sK)$ and
therefore $\cc(\sK)$ at a unique point.

Now let $\mb y\in(\R_\ge,\R_>)^\sK$ be an arbitrary point. Let
$I(\mb y)\in\sK$ be the set of zero coordinates of~$\mb y$, and
let $\sigma\in\Sigma$ be the cone corresponding to $I(\mb y)$. The
cones containing $\sigma$ constitute a fan $\st\sigma$ (called the
\emph{star} of~$\sigma$) in the quotient space $\R^n/\R^{I(\mb
y)}$. Its underlying simplicial complex is the \emph{link} $\lk
I(\mb y)$ of $I(\mb y)$ in~$\sK$. Now observe that the action of
$R_\Sigma$ on the set
\[
  \{(y_1,\ldots,y_m)\in(\R_\ge,\R_>)^\sK\colon y_i=0\text{ for }i\in I(\mb
  y)\}\cong(\R_\ge,\R_>)^{\lk I(\mb y)}
\]
coincides with the action of the group $R_{\st\sigma}$. Now we can
repeat the above arguments for the complete fan $\st\sigma$ and
the action of $R_{\st\sigma}$ on $(\R_\ge,\R_>)^{\lk I(\mb y)}$.
As the result, we obtain that every $R_\Sigma$-orbit intersects
$\cc(\sK)$ at a unique point.

To finish the proof of~(b) we consider the commutative diagram
\[
\begin{CD}
\zk @>>> U(\sK)\\
@VVV @VV\pi V\\
\cc(\sK) @>>> (\R_\ge,\R_>)^\sK
\end{CD}
\]
where the horizontal arrows are embeddings and the vertical ones
are projections onto the quotients of $\T^m$-actions. Note that
the projection $\pi$ commutes with the $R_\Sigma$-actions on
$U(\sK)$ and $(\R_\ge,\R_>)^\sK$, and the subgroups $R_\Sigma$ and
$\T^m$ of $(\C^\times)^m$ intersect trivially. It follows that
every $R_\Sigma$-orbit intersects the full preimage
$\pi^{-1}(\cc(\sK))=\zk$ at a unique point. Indeed, assume that
$\mb z$ and $r\mb z$ are in $\zk$ for some $\mb z\in U(\sK)$ and
$r\in R_\Sigma$. Then $\pi(\mb z)$ and $\pi(r\mb z)=r\pi(\mb z)$
are in $\cc(\sK)$, which implies that $\pi(\mb z)=\pi(r\mb z)$.
Hence, $\mb z=\mb t r\mb z$ for some $\mb t\in\T^m$. We may assume
that $\mb z\in(\C^\times)^m$, so that the action of both
$R_\Sigma$ and $\T^m$ is free (otherwise consider the action on
$U(\lk I(\mb z))$ where $I(\mb z)\in\sK$ is the set of zero
coordinates of~$\mb z$).
It follows that $\mb t r=\mathbf 1$, which implies that $r=\mathbf
1$, since $R_\Sigma$ and $\T^m$ intersect trivially.
\end{proof}

\begin{remark}
Our construction of a smooth structure on $\mathcal
Z_{\sK_\Sigma}$ depends on the geometry of~$\Sigma$. However, we
expect that the smooth structures coming from fans $\Sigma$ and
$\Sigma'$ are the same whenever the underlying simplicial
complexes $\sK_\Sigma$ and $\sK_{\Sigma'}$ are isomorphic.
Equivalently, the quotients $\mathcal Z_{\sK_\Sigma}/\T^m$ and
$\mathcal Z_{\sK_{\Sigma'}}/\T^m$ are diffeomorphic as manifolds
with corners
whenever $\sK_\Sigma=\sK_{\Sigma'}$. It is true in the polytopal
case (see also discussion in Section~4), and also for those fans
$\Sigma$ which are \emph{shellable}. (A shelling order allows us
to use an inductive argument, at each step extending a
diffeomorphism between two $(k-1)$-balls in the boundaries of
$k$-balls to the whole $k$-balls.)
\end{remark}

We do not know if Theorem~\ref{zksmooth} generalises to other
sphere triangulations:

\begin{question}
Describe the class of sphere triangulations $\sK$ for which the
moment-angle manifold $\zk$ admits a smooth structure.
\end{question}

\section{Complex-analytic structures}
Here we show that the even-dimensional moment-angle manifold $\zk$
corresponding to a complete simplicial fan admits a structure of a
complex manifold. The idea is to replace the action of
$\R^{m-n}_>$ on $U(\sK)$ (whose quotient is $\zk$) by a
holomorphic action of $\C^{\frac{m-n}2}$ on the same space.

In this section we assume that $m-n$ is even. We can always
achieve this by formally adding an `empty' one-dimensional cone to
$\Sigma$; this corresponds to adding a ghost vertex to~$\sK$, or
multiplying $\zk$ by a circle. The column of matrix $\Lambda_\R$
corresponding to the `empty' 1-cone is set to be zero. Set
$\ell:=\frac{m-n}2$.

We identify $\C^m$ (as a real vector space) with $\R^{2m}$ using
the map
\begin{equation}\label{crid}
  (z_1,\ldots,z_m)\mapsto(x_1,y_1,\ldots,x_m,y_m),
\end{equation}
where $z_k=x_k+iy_k$ for $k=1,\ldots,m$, and denote the
$\R$-linear map $\C^m\to\R^m$,
$(z_1,\ldots,z_m)\mapsto(x_1,\ldots,x_m)$ by~$\Re$.

\begin{construction}\label{psi}
Choose a linear map $\Psi\colon\C^\ell\to\C^m$ satisfying two
conditions:
\begin{itemize}
\item[(a)] $\Re\circ\Psi\colon\C^\ell\to\R^m$ is a
monomorphism.

\item[(b)] $\Lambda_\R\circ\Re\circ\Psi=0$.
\end{itemize}
This corresponds to choosing a complex structure and specifying a
complex basis in the real vector space
$\Ker\Lambda_\R\cong\R^{2\ell}$. Consider the following
commutative diagram:
\begin{equation}\label{cdiag}
\begin{CD}
  \C^\ell @>\Psi>> \C^m @>\Re>> \R^m @>\Lambda_\R>> \R^n\\
  @. @VV\exp V @VV\exp V @VV\exp V\\\
  @. (\C^\times )^m@>|\cdot|>> \R^m_> @>\exp\Lambda_\R>> \R^n_>
\end{CD}
\end{equation}
where the vertical arrows are the componentwise exponential maps,
and $|\cdot|$ denotes the map
$(z_1,\ldots,z_m)\mapsto(|z_1|,\ldots,|z_m|)$. Now set
\begin{equation}\label{csigma}
  C_{\Psi,\Sigma}:=\exp\Psi(\C^\ell)
  =\bigl\{\bigl(e^{\langle\psi_1,\mb w\rangle},\ldots,
  e^{\langle\psi_m,\mb w\rangle}\bigr)\in(\C^\times )^m\bigr\}
\end{equation}
where $\mb w=(w_1,\ldots,w_\ell)\in\C^\ell$, \ $\psi_i$ denotes
the $i$th row of the $m\times\ell$-matrix $\Psi=(\psi_{ij})$, and
$\langle\psi_i,\mb
w\rangle=\psi_{i1}w_1+\ldots+\psi_{i\ell}w_\ell$. Then
$C_{\Psi,\Sigma}\cong\C^\ell$ is a complex-analytic (but not
algebraic) subgroup in~$(\C^\times )^m$. It acts on $U(\sK)$ by
holomorphic transformations.
\end{construction}

\begin{example}\label{2torus}
Let $\sK$ be a simplicial complex on a two element set consisting
of the empty set only (that is, $\sK$ has two ghost vertices). We
therefore have $n=0$, $m=2$, $\ell=1$, and
$\Lambda_\R\colon\R^2\to0$ is a zero map. Let
$\Psi\colon\C\to\C^2$ be given by $z\mapsto(z,\alpha z)$ for some
$\alpha\in\C$, so that subgroup~\eqref{csigma} is
\[
  C=\{(e^z,e^{\alpha z})\}\subset(\C^\times )^2.
\]
Condition~(b) of Construction~\ref{psi} is void, while
condition~(a) is equivalent to that $\alpha\notin\R$. Then
$\exp\Psi\colon\C\to(\C^\times )^2$ is an embedding, and the
quotient $(\C^\times )^2/C$ with the natural complex structure is
a complex torus $T^2_\C$ with parameter $\alpha\in\C$:
\[
  (\C^\times )^2/C\cong\C/(\Z\oplus\alpha\Z)=T^2_\C(\alpha).
\]

If we start with the empty simplicial complex on the set of
$2\ell$ elements (so that $n=0$, $m=2\ell$), we may obtain an
arbitrary compact complex $\ell$-dimensional torus $T_\C^{2\ell}$
as the quotient $(\C^\times )^{2\ell}/C_{\Psi,\Sigma}$,
see~\cite[Th.~1]{meer00}.
\end{example}

\begin{theorem}\label{zkcomplex}
Let $\Sigma$ be a complete simplicial fan in~$\R^n$ with $m$
one-dimensional cones, and let $\sK=\sK_\Sigma$ be its underlying
simplicial complex. Assume that $m-n=2\ell$ and let
$C_{\Psi,\Sigma}$ be a subgroup of $(\C^\times)^m$ defined
by~\eqref{csigma}. Then
\begin{itemize}
\item[(a)]
the holomorphic action of the group $C_{\Psi,\Sigma}$ on $U(\sK)$
is free and proper, and the quotient $U(\sK)/C_{\Psi,\Sigma}$ has
a structure of a compact complex manifold of complex dimension
$m-\ell$;

\item[(b)] there is a $\T^m$-equivariant diffeomorphism between $U(\sK)/C_{\Psi,\Sigma}$ and
$\zk$ defining a complex structure on $\zk$, in which $\T^m$ acts
by holomorphic transformations.
\end{itemize}
\end{theorem}
\begin{proof}
We first prove statement (a). The isotropy subgroups of the
$(\C^\times)^m$-action on $U(\sK)$ are of the form
$(\C^\times,1)^I$ for $I\in\sK$. In order to show that the
subgroup $C_{\Psi,\Sigma}$ acts freely we need to check that
$C_{\Psi,\Sigma}$ intersects every $(\C^\times)^m$-isotropy
subgroup only at the unit. Since $C_{\Psi,\Sigma}$ embeds into
$\R^m_>$ by~\eqref{cdiag}, it enough to check that the image of
$C_{\Psi,\Sigma}$ in $\R^m_>$ intersects the image of
$(\C^\times,1)^I$ only at the unit. The former image is $R_\Sigma$
and the latter image is $(\R_>,1)^I$; the triviality of their
intersection follows from Theorem~\ref{zksmooth}~(a).

Now we prove the properness of this action. Consider the
projection $\pi\colon U(\sK)\to(\R_\ge,\R_>)^\sK$ onto the
quotient of the $\T^m$-action, and the commutative square
\[
\begin{CD}
  C_{\Psi,\Sigma}\times U(\sK) @>h_\C>> U(\sK)\times U(\sK)\\
  @VVi\times\pi V @VV\pi\times\pi V\\
  R_\Sigma\times (\R_\ge,\R_>)^\sK@>h_\R>> (\R_\ge,\R_>)^\sK\times(\R_\ge,\R_>)^\sK
\end{CD},
\]
where $h_\C$ and $h_\R$ denote the group action maps, and $i\colon
C_{\Psi,\Sigma}\to R_\Sigma$ is the isomorphism given by the
restriction of $|\cdot|\colon (\C^\times )^m\to \R_>^m$. The
preimage $h^{-1}_\C(V)$ of any compact subset $V\in U(\sK)\times
U(\sK)$ is a closed subset in the set $W=(i\times\pi)^{-1}\circ
h_\R^{-1}\circ(\pi\times\pi)(V)$. The image $\pi\times\pi(V)$ is
compact, the action of $R_\Sigma$ on $(\R_\ge,\R_>)^\sK$ is proper
by the same argument as used in the proof of
Theorem~\ref{zksmooth}~(a), and the map $i\times \pi$ is proper,
since it is the quotient projection of a compact group action.
Hence, $W$ is a compact subset in $C_{\Psi,\Sigma}\times U(\sK)$,
and $h^{-1}_\C(V)$ is compact as a closed subset in~$W$.

The complex group $C_{\Psi,\Sigma}$ acts holomorphically, freely
and properly on the complex manifold $U(\sK)$, therefore by the
complex analogue of~\cite[Th.~7.10]{lee00}, the orbit space admits
a structure of a complex manifold.

Like in the real situation of Section~2, it is possible to
construct an atlas of $U(\sK)/C_{\Psi,\Sigma}$ explicitly. Since
the action of $C_{\Psi,\Sigma}$ on the quotient
$U(\sK)/\T^m=(\R_\ge,\R_>)^\sK$ coincides with the action of
$R_\Sigma$ on the same space, the quotient of
$U(\sK)/C_{\Psi,\Sigma}$ by the action of $\T^m$ has exactly the
same structure of a smooth manifold with corners as the quotient
of $U(\sK)/R_\Sigma$ by $\T^m$ (see the proof of
Theorem~\ref{zksmooth}). This structure is determined by the atlas
$\{q_I\colon(\R_\ge,\R_>)^I/R_\Sigma\to\R^n_\ge\}$, which lifts to
a covering of $U(\sK)/C_{\Psi,\Sigma}$ by open subsets
$(\C,\C^\times)^I/C_{\Psi,\Sigma}$. For every $I\in\sK$, the
subset $(\C,\T)^I\subset(\C,\C^\times)^I$ intersects any orbit of
the $C_{\Psi,\Sigma}$-action on $(\C,\C^\times)^I$ transversely at
a single point. Therefore, every
$(\C,\C^\times)^I/C_{\Psi,\Sigma}\cong(\C,\T)^I$ acquires a
structure of a complex manifold. Since
$(\C,\C^\times)^I\cong\C^n\times(\C^\times)^{m-n}$, and the action
of $C_{\Psi,\Sigma}$ on the $(\C^\times)^{m-n}$ factor is free,
the complex manifold $(\C,\C^\times)^I/C_{\Psi,\Sigma}$ is the
total space of a holomorphic $\C^n$-bundle over the compact
complex torus $T_\C^{2\ell}=(\C^\times)^{m-n}/C_{\Psi,\Sigma}$
(see Example~\ref{2torus}). Writing trivialisations of these
$\C^n$-bundles for every~$I$, we obtain a holomorphic atlas for
$U(\sK)/C_{\Psi,\Sigma}$.

The proof of statement (b) follows the lines of the proof of
Theorem~\ref{zksmooth}~(b). We need to show that every
$C_{\Psi,\Sigma}$-orbit intersects $\zk\subset U(\sK)$ at a single
point. First we show that the $C_{\Psi,\Sigma}$-orbit of every
point in $U(\sK)/\T^m$ intersects $\zk/\T^m=\cc(\sK)$ at a single
point; this follows from the fact that the actions of
$C_{\Psi,\Sigma}$ and $R_\Sigma$ coincide on $U(\sK)/\T^m$. Then
we show that every $C_{\Psi,\Sigma}$-orbit intersects the full
preimage $\pi^{-1}(\cc(\sK))$ at a single point using the fact
that $C_{\Psi,\Sigma}$ and $\T^m$ have trivial intersection in
$(\C^\times)^m$.
\end{proof}

\begin{remark}
Unlike the smooth structure, the complex structure on $\zk$
depends on both the geometry of $\Sigma$ and the choice of~$\Psi$
in Construction~\ref{psi}. The fact that the choice of $\Psi$
affects the complex structure is already clear from
Example~\ref{2torus}.
\end{remark}

\begin{example}[Hopf manifold]\label{hopf}
Let $\Sigma$ be a complete simplicial fan in $\R^n$ whose cones
are generated by all proper subsets of the set of $n+1$ vectors
$\mb e_1,\ldots,\mb e_n,-\mb e_1-\ldots-\mb e_n$. To make $m-n$
even we also add one `empty' 1-cone. We therefore have $m=n+2$,
$\ell=1$. Then $\Lambda_\R\colon\R^{n+2}\to\R^n$ is given by the
matrix $(\mathbf 0\; I\; {-}\mathbf 1)$, where $I$ is the unit
$n\times n$ matrix, and $\mathbf 0$, $\mathbf 1$ are the
$n$-columns of zeros and units respectively.

We have that $\sK$ is the boundary of an $n$-dimensional simplex
with $n+1$ vertices and 1 ghost vertex, $\zk\cong S^1\times
S^{2n+1}$, and $U(\sK)=\C^\times \times(\C^{n+1}\setminus\{0\})$.
Let $\Psi\colon\C\to\C^{n+2}$ be given by $z\mapsto(z,\alpha
z,\ldots,\alpha z)$ for some $\alpha\in\C$. Like in
Example~\ref{2torus}, conditions of Construction~\ref{psi} imply
that $\alpha$ is not a real number, and $\exp\Psi$ embeds $\C$ as
the following subgroup in $(\C^\times )^{n+2}$:
\[
  C=\bigl\{(e^z,e^{\alpha z},\ldots,e^{\alpha z})\colon z\in\C\}\subset(\C^\times )^{n+2}.
\]
By Theorem~\ref{zkcomplex}, $\zk$ acquires a complex structure as
the quotient $U(\sK)/C$:
\[
  \C^\times \times\bigl(\C^{n+1}\setminus\{0\}\bigr)\bigl/\;\{(t,\mb
  w)\!\sim(e^zt,e^{\alpha z}\mb w)\}
  \cong\bigl(\C^{n+1}\setminus\{0\}\bigr)\bigl/\;\{\mb w\!\sim
  e^{2\pi i\alpha}\mb w\},
\]
where $t\in\C^\times $, $\mb w\in\C^{n+1}\setminus\{0\}$. The
latter is the quotient of $\C^{n+1}\setminus\{0\}$ by the action
of $\Z$ generated by the multiplication by $e^{2\pi i\alpha}$. It
is known as a \emph{Hopf manifold}.
\end{example}

Theorem~\ref{zkcomplex} may be generalised to a wider class of
manifolds, namely, to \emph{partial quotients} of~$\zk$ in the
sense of~\cite[\S7.5]{bu-pa02}. Assume there exist a primitive
sublattice $N\subset\Z^m$ of rank $k$ such that $\Lambda_\R N=0$.
We denote the corresponding $k$-torus by $T(N)\subset \T^m$.

\begin{proposition}
The torus $T(N)$ acts freely on $U(\sK)$.
\end{proposition}
\begin{proof}
As in the proof of the freeness of the $R_\Sigma$-action in
Theorem~\ref{zksmooth}, all $\T^m$-isotropy subgroups are of the
form $(\T,1)^I$ for $I\in\sK$. Since primitive sublattices
$(\Z,0)^I$ and $N$ in $\Z^m$ intersect trivially, the intersection
of the corresponding tori $(\T,1)^I$ and $T(N)$ is also trivial.
\end{proof}

Let $N_\R\subset\R^m$ be the linear subspace generated by~$N$, and
let $\rho\colon \R^m\to\R^m/N_\R$ be the quotient projection. Let
$T_\C(N)=N\otimes_{\Z}\C^\times\cong(\C^\times)^k$ be the
algebraic torus corresponding to~$N$. Assuming that $2\ell=m-n-k$
is even, we have the following generalisation of
Construction~\ref{psi}.

\begin{construction}\label{psi_part}
Choose a linear map $\Omega\colon \C^\ell\to\C^m$ satisfying the
following two conditions:
\begin{itemize}
\item[(a)] $\rho\circ\Re\circ\,\Omega$ is a monomorphism.
\item[(b)] $\Lambda_\R\circ\Re\circ\,\Omega=0$.
\end{itemize}

It follows from (a) that the subgroups $T_\C(N)$ and
$\exp\Omega(\C^\ell)$ have trivial intersection
in~$(\C^\times)^m$, and therefore we may define a subgroup
\begin{equation} \label{csigma_part}
  C_{\Omega,\Sigma}(N):=T_\C(N)\times\exp\Omega(\C^\ell)
  \subset(\C^\times)^m.
\end{equation}
\end{construction}

\begin{theorem}\label{zkcomplex_part}
Let $\Sigma$ and $\sK$ be as in Theorem~\ref{zkcomplex}, and
choose $N$ and $C_{\Omega,\Sigma}(N)$ as in
Construction~\ref{psi_part}. Then
\begin{itemize}
\item[(a)] the holomorphic action of the group
$C_{\Psi,\Sigma}(N)$ on $U(\sK)$ is free and proper, and the
quotient $U(\sK)/C_{\Omega,\Sigma}(N)$ has a structure of a
compact complex manifold of complex dimension $m-\ell-k$;

\item[(b)] there is a $\T^m$-equivariant diffeomorphism between
$U(\sK)/C_{\Omega,\Sigma}(N)$ and $\zk/T(N)$ defining a complex
structure on the quotient $\zk/T(N)$, in which $\T^m/T(N)$ acts by
holomorphic transformations.
\end{itemize}
\end{theorem}

The proof is similar to that of Theorem~\ref{zkcomplex} and is
omitted.

\begin{example}
1. The case $k=0$ gives back Theorem~\ref{zkcomplex}.

2. Let $k=1$ and take $N$ to be the diagonal sublattice of rank
one in~$\Z^m$. The condition $\Lambda_\R N=0$ implies that the
vectors $\mb a_1,\ldots,\mb a_m$ sum up to zero, which can always
be achieved by rescaling them (as $\Sigma$ is a complete fan). As
the result, we obtain a complex structure on the quotient of $\zk$
by the diagonal subgroup in~$\T^m$, provided that $m-n$ is odd. In
the polytopal case $\sK=\sK_P$ this gives the construction of
\emph{LVM manifolds} from~\cite{meer00} (see also~\cite{bo-me06}
and the discussion in the next section).

3. Let $k=m-n$. Then $\Lambda_\R N=0$ implies that the whole
$\Ker\Lambda_\R$ is generated by a primitive lattice. In this case
$\ell=0$ and $C_{\Omega,\Sigma}(N)=T_\C(N)$,
see~\eqref{csigma_part}. The quotient $U(\sK)/T_\C(N)=\zk/T(N)$ is
the \emph{toric variety} corresponding to the rational
fan~$\Sigma$, see Section~\ref{sectoric}.
\end{example}

\section{Moment-angle manifolds from polytopes}
Normal fans of simple convex polytopes $P$ constitute an important
class of complete simplicial fans. The results of the previous
section have interesting specifications in the polytopal case.
Polytopal moment-angle manifolds $\zp$ admit an explicit
description as intersections of quadratic hypersurfaces in~$\C^m$.
This description originates from~\cite[Constr.~6.8]{bu-pa02}, and
explicit quadratic equations for $\zp$ were written
in~\cite[(3.3)]{b-p-r07}.

Since polytopal moment-angle manifolds $\zp$ were shown to be
complex in~\cite{bo-me06}, we find it interesting and important to
relate explicitly the quadratic description of $\zp$
from~\cite{bu-pa02} and~\cite{b-p-r07} to the quadratic
description of LVM-manifolds from~\cite{meer00}
and~\cite{bo-me06}. It helps to understand better how the complex
structures on the polytopal moment-angle manifolds $\zp$ fit the
more general framework of the previous section.

Let $P$ be an $n$-dimensional convex polytope given as an
intersection of $m$ halfspaces in a Euclidean space $\R^n$ with
the scalar product $\langle\;,\:\rangle$:
\begin{equation}\label{ptope}
  P:=\bigl\{\mb x\in\R^n\colon\langle\mb a_i,\mb x\rangle+b_i\ge0\quad\text{for }
  i=1,\ldots,m\bigr\},
\end{equation}
where $\mb a_i\in\R^n$ and $b_i\in\R$. We refer to the right hand
side of~\eqref{ptope} as a \emph{presentation} of $P$ by a system
of linear inequalities. An inequality $\langle\mb a_i,\mb
x\rangle+b_i\ge0$ is \emph{redundant} if it can be removed from
the presentation without changing~$P$. A presentation of $P$
without redundant inequalities (an \emph{irredundant}
presentation) is unique; however for the reasons explained below
we allow redundant inequalities and therefore indeterminacy in the
presentation of~$P$. We consider the hyperplanes
\[
  H_i:=\bigl\{\mb x\in\R^n\colon\langle\mb a_i,\mb x\rangle+b_i=0\bigr\}
\]
for $i=1,\ldots,m$, and refer to a presentation of $P$ as
\emph{generic} if at most $n$ hyperplanes $H_i$ meet at every
point of~$P$. The existence of a generic presentation implies that
$P$ is \emph{simple}, that is, there exactly $n$ facets meet at
every vertex of~$P$. A generic presentation may contain redundant
inequalities, but for every such inequality the intersection of
the corresponding hyperplane $H_i$ with $P$ is empty (that is, the
inequality is strict for every $\mb x\in P$).

\begin{construction}\label{dist}
Let $A_P$ be the $m\times n$ matrix of row vectors $\mb a_i$ and
$\mb b_P$ be the column vector of scalars $b_i\in\R^n$.
Consider the affine map
\[
  i_P\colon \R^n\to\R^m,\quad i_P(\mb x)=A_P\mb x+\mb b_P.
\]
It is monomorphic onto a certain $n$-dimensional plane in $\R^m$,
and $i_P(P)$ is the intersection of this plane with~$\R_\ge^m$.

We define the space $\mathcal Z_P$ from the commutative diagram
\begin{equation}\label{cdiz}
\begin{CD}
  \mathcal Z_P @>i_Z>>\C^m\\
  @VVV\hspace{-0.2em} @VV\mu V @.\\
  P @>i_P>> \R^m_\ge
\end{CD}
\end{equation}
where $\mu(z_1,\ldots,z_m)=(|z_1|^2,\ldots,|z_m|^2)$. The latter
map may be thought of as the quotient map for the coordinatewise
action of the standard torus $\T^m$ on~$\C^m$. Therefore, $\T^m$
acts on $\zp$ with quotient $P$, and $i_Z$ is a $\T^m$-equivariant
embedding.

Now choose an $(m-n)\times m$ matrix $\varGamma=(\gamma_{jk})$
whose rows form a basis of linear relations between the
vectors~${\mb a}_i$, \ $i=1,\ldots,m$. That is, the corresponding
map $\varGamma\colon\R^m\to\R^{m-n}$ completes the short exact
sequence
\begin{equation}\label{pexse}
  0\longrightarrow \R^n\stackrel{A_P}{\longrightarrow}\R^m
  \stackrel{\varGamma}{\longrightarrow}\R^{m-n}\longrightarrow 0.
\end{equation}
Then we may write the image of $P$ under $i_P$ by linear equations
as
\begin{equation}\label{cequations}
  i_P(P)=\bigl\{\mb y\in\R^m_\ge\colon \varGamma\mb y=\varGamma\mb
  b_P,\quad\text{for }1\le i\le m\bigr\}.
\end{equation}
Comparing this with diagram~\eqref{cdiz} we finally obtain that
$\zp$ embeds into $\C^m$ as the set of common zeros of $m-n$ real
quadratic equations
\begin{equation}\label{zpqua}
  \sum_{k=1}^m\gamma_{jk}|z_k|^2=\sum_{k=1}^m\gamma_{jk}b_k,\;\text{ for }
  1\le j\le m-n.
\end{equation}
%
\end{construction}

The following properties of $\zp$ are immediate consequences of
its construction.

\begin{proposition}\label{easyzp}
\begin{itemize}
\item[(1)] Given a point $z\in\zp$, the $k$th coordinate of
$i_Z(z)\in\C^m$ vanishes if and only if $z$ projects onto a point
$\mb x\in P$ such that $\mb x\in H_k$.

\item[(2)] Adding a redundant inequality to~\eqref{ptope} results
in multiplying $\zp$ by a circle.
\end{itemize}
\end{proposition}

\begin{theorem}\label{smgen}
The intersection of quadrics~\eqref{zpqua} defining $\zp$ is
nondegenerate (transverse) if and only if
presentation~\eqref{ptope} is generic.
\end{theorem}
\begin{proof}
For simplicity, we do not distinguish between $\zp$ and its image
$i_Z(\zp)$ in $\C^m$ in this proof. Using
identification~\eqref{crid} we observe that the gradients of the
$m-n$ quadratic forms in~\eqref{zpqua} at a point $\mb
z=(x_1,y_1,\ldots,x_m,y_m)\in\zp$ are
\begin{equation}\label{grve}
  2\left(\gamma_{j1}x_1,\,\gamma_{j1}y_1,\,\dots,\,\gamma_{jm}x_m,\,\gamma_{jm}y_m\right),\quad
  1\leq j\leq m-n.
\end{equation}
These gradients form the rows of the $(m-n)\times 2m$ matrix
$2\varGamma\varDelta$, where
\[
\varDelta\;=\;
\begin{pmatrix}
  x_1&y_1& \ldots & 0  &0 \\
  \vdots & \vdots & \ddots & \vdots & \vdots\\
  0  & 0 & \ldots &x_m &y_m
\end{pmatrix}.
\]
Assume that for the chosen point $\mb z\in\zp$ we have
$z_{j_1}=\ldots=z_{j_k}=0$, and the other complex coordinates are
nonzero. Then the rank of the gradient matrix
$2\varGamma\varDelta$ at $\mb z$ is equal to the rank of the
$(m-n)\times(m-k)$ matrix $\varGamma'$ obtained by deleting the
columns $j_1,\ldots,j_k$ from~$\varGamma$.

By Proposition~\ref{easyzp}, the point $\mb z$ projects onto $\mb
x\in H_{j_1}\cap\ldots\cap H_{j_k}$.

If presentation~\eqref{ptope} is generic, then the hyperplanes
$H_{j_1},\ldots,H_{j_k}$ meet transversely (in particular, $k\le
n$). It follows that the rank of the $k\times n$ submatrix $A'$ of
$A_P$ formed by the rows $\mb a_{j_1},\ldots,\mb a_{j_k}$ is~$k$.
This implies that $\varGamma'$ has rank $m-n$
(see~\cite[Lemma~2.18]{b-p-r07}), and therefore~\eqref{zpqua} is
nondegenerate at~$\mb z$.

On the other hand, if~\eqref{ptope} is not generic, then we
consider a point $\mb z\in\zp$ which projects to the intersection
of $k>n$ hyperplanes $H_i$. Then for this $\mb z$ the matrix
$\varGamma'$ has $m-k<m-n$ columns, and therefore its rank is less
than $m-n$. It follows that~\eqref{zpqua} is degenerate at~$\mb
z$.
\end{proof}

\begin{corollary}[{\cite[Cor.~3.9]{bu-pa02}, \cite[Lemma~4.2]{b-p-r07}}]
$\zp$ is a smooth manifold of dimension $m+n$. Moreover, the
embedding $i_z\colon\zp\to\C^m$ is $T^m$-equivariantly framed by
any choice of matrix $\varGamma$ in~\eqref{pexse}.
\end{corollary}

Assume now that~\eqref{ptope} is generic; in particular, $P$ is
simple. Set
\[
  F_i:=H_i\cap P=\bigl\{\mb x\in P\colon\langle\mb a_i,\mb x\rangle+b_i=0\bigr\}
\]
for $i=1,\ldots,m$. Note that $F_i$ is either empty (if the $i$th
inequality is redundant), or is a facet of~$P$. The \emph{normal
fan} $\Sigma_P$ of $P$ consists of cones spanned by the sets of
vectors $\{\mb a_{j_1},\ldots,\mb a_{j_k}\}$ for which the
intersection $F_{j_1}\cap\ldots\cap F_{j_k}$ is nonempty. It is a
complete simplicial fan in~$\R^n$. We formally add to $\Sigma_P$
`empty 1-dimensional cones' corresponding to redundant
inequalities (or `empty facets'~$F_i$). Then the underlying
simplicial complex $\sK_P=\sK_{\Sigma_P}$ is on the set $[m]$ and
has a ghost vertex for every redundant inequality.

\begin{proposition}
The manifolds $\mathcal Z_{\sK_P}$ and $\mathcal Z_P$ defined
by~\eqref{defzk} and~\eqref{cdiz} are $\T^m$-equivariantly
homeomorphic.
\end{proposition}
\begin{proof}
The idea is to write both $\mathcal Z_{\sK_P}$ and $\mathcal Z_P$
as a certain identification space $P\times\T^m/\!\sim\:$, as
in~\cite{da-ja91}.

We consider the map $\R_\ge\times\T\to\C$ defined by $(y,t)\mapsto
yt$. Taking product we obtain a map $\R^m_\ge\times\T^m\to\C^m$.
The preimage of a point $\mb z\in\C^m$ under this map is $\mb
y\times(\T,1)^{I(\mb z)}$, where $y_i=|z_i|$ for $1\le i\le m$, \
$I(\mb z)\subset[m]$ is the set of zero coordinates of~$\mb z$,
and $(\T,1)^{I(\mb z)}\subset\T^m$ is the coordinate subgroup
defined by~\eqref{xwi}. Therefore, $\C^m$ can be identified with
the quotient space $\R^m_\ge\times\T^m/{\sim\:}$, where $(\mb
y,\mb t_1)\sim(\mb y,\mb t_2)$ if $\mb t_1^{-1}\mb
t_2\in(\T,1)^{I(\mb y)}$.

From~\eqref{cdiz} we obtain a similar description of $\zp$ as an
identification space. Namely, given $p\in P$, set
$I_p=\{i\in[m]\colon p\in F_i\}$ (the set of facets containing
$p$). Then $\zp$ is $\T^m$-equivariantly homeomorphic to
\begin{equation}\label{defzp}
  P\times\T^m/{\sim\:}\text{ where }(p,\mb t_1)\sim(p,\mb t_2)\text{ if }\mb t_1^{-1}\mb t_2\in
  (\T,1)^{I_p}.
\end{equation}

On the other hand, $\mathcal Z_{\sK_P}$ can be defined from a
diagram similar to~\eqref{cdiz}:
\[
\begin{CD}
  \mathcal Z_{\sK_P} @>>>\D^m\\
  @VVV\hspace{-0.2em} @VV\mu V @.\\
  \cc(\sK_P) @>>> \I^m
\end{CD}
\]
We have $\D^m\cong\I^m\times\T^m/{\sim\:}$ by restriction of the
corresponding construction for~$\C^m$. Under the identification of
$P$ with $\cc(\sK_P)\subset\I^m$ a point $p\in P$ is mapped to a
point in $\I^m$ whose set of zero coordinates is exactly $I_p$
(see the details in~\cite[\S4.2]{bu-pa02}). Therefore, $\mathcal
Z_{\sK_P}$ can be also identified with~\eqref{defzp}.
\end{proof}

\begin{remark}
It is easy to observe that $\zp\subset U(\sK_P)$. An argument
similar to that outlined in~\cite[Lemma~0.8]{bo-me06} can be used
to show that every orbit of the free action of $R_{\Sigma_P}$
intersects $\zp$ transversely at a single point (see
also~\cite[Appendix~1]{guil94}). Therefore, in the polytopal case
the intersection of quadrics~\eqref{zpqua} can be used instead of
the moment-angle complex~$\mathcal Z_{\sK_P}$ as a transverse set
to the orbits of the $R_{\Sigma_P}$-action. This also implies that
the smooth structure on $\mathcal Z_{\sK_P}$ defined by
Theorem~\ref{zksmooth} is equivalent to the smooth structure on
the intersection of quadrics~\eqref{zpqua}. Another way to see
that these two smooth structures coincide is to identify both with
the smooth structure on $\zp$ from~\cite[Lemma~6.2]{bu-pa02}.
\end{remark}

\begin{corollary}
The $\T^m$-equivariant homeomorphism type of the manifold $\zp$
depends only on the combinatorial type of the polytope~$P$.
\end{corollary}

We therefore refer to either $\mathcal Z_{\sK_P}$ or $\mathcal
Z_P$ as the \emph{moment-angle manifold corresponding to simple
polytope}~$P$, or shortly \emph{polytopal moment-angle manifold}.

\begin{corollary}\label{zpcomplex}
The moment-angle manifold $\zp$ admits a complex structure as the
quotient $U(\sK_P)/C_{\Psi,\Sigma_P}$ if $\dim\zp=m+n$ is even. If
$m+n$ is odd, then the product $\zp\times S^1$ admits such a
structure.
\end{corollary}
\begin{proof}
We identify $\zp$ with $\mathcal Z_{\sK_P}$ and equip the latter
with a complex analytic structure of the quotient
$U(\sK_P)/C_{\Psi,\Sigma_P}$ using Theorem~\ref{zkcomplex},
provided that $m+n$ is even. If $m+n$ is odd, we add one redundant
inequality to~\eqref{ptope} (the simplest one is $1\ge0$, with
$\mb a_i=0$ and $b_i=1$); then $m$ increases by one, $n$ does not
change, and $\zp$ multiplies by a circle.
\end{proof}

The existence of complex structures on intersections of quadrics
similar to~\eqref{zpqua} was established in~\cite{bo-me06} as a
consequence of the construction of LVM-manifolds in~\cite{lo-ve97}
and~\cite{meer00}. Namely, in~\cite{bo-me06} there were considered
transverse intersections of \emph{homogeneous} quadrics in~$\C^m$
with a unit sphere:
\begin{equation}\label{bmlink}
  \mathcal L:=\left\{\begin{array}{ll}
  \mb z\in\C^m\colon&\sum_{k=1}^m d_{jk}|z_k|^2=0,\quad d_{jk}\in\R,\; 1\le j\le m-n-1,\\[2mm]
  &\sum_{k=1}^m|z_k|^2=1.
  \end{array}\right\}
\end{equation}
The transversality condition (guaranteeing that~\eqref{bmlink} is
a smooth manifold of dimension~$m+n$) translates via the
\emph{Gale transform} to the condition that the quotient of
$\mathcal L$ by the standard action of $\T^m$ is an
$n$-dimensional simple convex polytope
(see~\cite[Lemma~0.12]{bo-me06}; compare Theorem~\ref{smgen}
above). Such $\mathcal L$ were called \emph{links}
in~\cite{bo-me06}.

The underlying smooth manifold $\mathcal N$ of an \emph{LVM
manifold} is obtained by dropping the last equation
in~\eqref{bmlink} and considering the intersection of the
remaining $m-n-1$ homogeneous quadrics in the projective space $\C
P^{m-1}$. The complex structure comes from identifying this
intersection of quadrics with the orbit space of a free action of
$\C^\ell$ on an open subset in~$\C P^{m-1}$. By the construction,
$\mathcal N$ is the quotient of $\mathcal L$ by a free action of a
circle. On the other hand, according to~\cite[Th.~12.2]{bo-me06},
every link~$\mathcal L$ admits a complex structure as an
LVM-manifold if its dimension $m+n$ is even; otherwise $\mathcal
L\times S^1$ admits such a structure.

To relate~\eqref{zpqua} to~\eqref{bmlink} we note the following.
Assume for simplicity that all redundant inequalities
in~\eqref{zpqua} are of the form $1\ge0$. The rows of the
coefficient matrix $\varGamma=(\gamma_{jk})$ of~\eqref{zpqua} form
a basis in the space of linear relations between the vectors $\mb
a_1,\ldots,\mb a_m$ in~\eqref{ptope}. Since these vectors are
determined only up to a positive multiple, we may assume that
$|\mb a_i|=1$ if the $i$th inequality is irredundant. Since $P$ is
a convex polytope, one of the relations between the $\mb a_i$'s is
$\sum_{i=1}^m(\mathop{\mathrm{vol}}F_i)\mb a_i=\mathbf 0$, where
$\mathop{\mathrm{vol}}F_i\ge0$ is the volume of the facet~$F_i$.
By rescaling the vectors $\mb a_i$ again we may achieve that
$\sum_{i=1}^m\mb a_i=\mathbf 0$ (this still leaves one scaling
parameter free), and choose the coefficients of this relation as
the last row of~$\varGamma$, that is, $\gamma_{m-n,k}=1$ for $1\le
k\le m$. Then the last of the quadrics in~\eqref{zpqua} takes the
form $\sum_{k=1}^m|z_k|^2=\sum_{k=1}^mb_k$. Summing up all $m$
inequalities of~\eqref{ptope}, using the fact that
$\sum_{i=1}^m\mb a_i=\mathbf 0$ and noting that at least one of
the inequalities is strict for some~$\mb x\in\R^n$, we obtain
$\sum_{k=1}^mb_k>0$. Using up the last free scaling parameter we
achieve that $\sum_{k=1}^mb_k=1$. Then the last quadric
in~\eqref{zpqua} becomes $\sum_{k=1}^m|z_k|^2=1$, which coincides
with the last equation in~\eqref{bmlink}. Subtracting this
equation with an appropriate coefficient from the other equations
in~\eqref{zpqua} we finally transform~\eqref{zpqua}
to~\eqref{bmlink}.

Since the intersection of quadrics~\eqref{zpqua} can be identified
with a link~\eqref{bmlink}, the above cited
result~\cite[Th.~12.2]{bo-me06} can be used to equip $\zp$ with a
complex structure. However, the method of
Corollary~\ref{zpcomplex} is somewhat more direct; it does not
require the passage to projectivisations and LVM-manifolds.

\section{Holomorphic bundles over toric varieties
and Hodge numbers}\label{sectoric}
A \emph{toric variety} is a
normal algebraic variety $X$ on which an algebraic torus
$(\C^\times)^n$ acts with a dense orbit. As is well known in
algebraic geometry, toric varieties are classified by rational
fans~\cite{dani78}. Under this correspondence, complete fans give
rise to complete varieties (compact in the usual topology), normal
fans of polytopes to projective varieties, regular fans to
nonsingular varieties, and simplicial fans to varieties with mild
(orbifold-type) singularities.

A construction due to several authors, which is now often referred
to as the `Cox construction'~\cite{cox95}, identifies a toric
variety $X_\Sigma$ corresponding to a rational simplicial fan
$\Sigma$ in $\R^n$ with the (geometric) quotient of
$U(\sK_\Sigma)$ by the action of a certain $(m-n)$-dimensional
algebraic subgroup $G_\Sigma$ in~$(\C^\times )^m$. (General toric
varieties also arise as \emph{categorical} quotients, but we shall
not need this here.) The Cox construction may be also regarded as
an algebraic version of the construction of Hamiltonian toric
manifolds via \emph{symplectic reduction}~\cite{guil94}.

In the case of rational simplicial polytopal fans $\Sigma_P$ a
construction of~\cite{me-ve04} identifies the corresponding
projective toric variety $X_P$ as the base of a holomorphic
principal \emph{Seifert fibration}, whose total space is the
moment-angle manifold~$\zp$ equipped with a complex structure of
an LVM-manifold, and fibre is a compact complex torus of complex
dimension $\ell=\frac{m-n}2$. (Seifert fibrations are
generalisations of holomorphic fibre bundles to the case when the
base is an orbifold.) If $X_P$ is a nonsingular projective toric
variety, then there is a holomorphic free action of a complex
$\ell$-dimensional torus $T^{2\ell}_\C$ on $\zp$ with
quotient~$X_P$.

Here we generalise the construction of~\cite{me-ve04} to rational
complete simplicial fans (not necessarily polytopal). By an
application of the Borel spectral sequence to the resulting
holomorphic principal fibre bundle $\zk\to X_\Sigma$ we derive
some information about the Hodge numbers of complex structures on
moment-angle manifolds.

In this section we assume that the fan $\Sigma$ is complete,
simplicial and rational. We choose as $\mb a_1,\ldots,\mb a_m$ the
primitive integral generators of the 1-dimensional cones.

\begin{construction}[`Cox construction']
Let $\Lambda_\C\colon\C^m\to\C^n$ be the complexification of
map~\eqref{lambdar}, and consider its complex exponential:
\[
\begin{aligned}
  \exp\Lambda_\C\colon(\C^\times)^m&\to(\C^\times)^n,\\
  (z_1,\ldots,z_m)&\mapsto
  \Bigr(\prod_{i=1}^mz_i^{a_{i1}},\ldots,\:
  \prod_{i=1}^mz_i^{a_{in}}\Bigl)
\end{aligned}
\]
(here we use that the fan is rational and the vectors $\mb
a_i=(a_{i1},\ldots,a_{in})$ are integral; otherwise the map above
is not defined). Set $G_\Sigma:=\Ker(\exp\Lambda_\C)$. This is an
$(m-n)$-dimensional algebraic subgroup in $(\C^\times)^m$, hence,
it is isomorphic to a product of $(\C^\times)^{m-n}$ and a finite
group, where the finite group is trivial if the fan is regular.
The group $G_\Sigma$ acts almost freely (with finite isotropy
subgroups) on the open set $U(\sK_\Sigma)$; moreover, this action
is free if $\Sigma$ is a regular fan. This is proved in the same
way as the freeness of the action of $R_\Sigma$ on $U(\sK_\Sigma)$
in Theorem~\ref{zksmooth}~(a).

The \emph{toric variety} associated with the fan $\Sigma$ is the
quotient $X_\Sigma:=U(\sK_\Sigma)/G_\Sigma$. It is a complex
algebraic variety of dimension~$n$. The variety $X_\Sigma$ is
nonsingular whenever $\Sigma$ is regular; otherwise it has
orbifold-type singularities (i.e. $X_\Sigma$ is locally isomorphic
to a quotient of $\C^n$ by a finite group). The quotient algebraic
torus $(\C^\times)^m/G_\Sigma\cong(\C^\times)^n$ acts on
$X_\Sigma$ with a dense orbit.

The variety $X_\Sigma$ is projective if and only if $\Sigma$ is
the normal fan of a convex polytope $P$; in this case we shall
denote the variety by~$X_P$.
\end{construction}

Assume that $m-n$ is even by adding an empty 1-cone to $\Sigma$ if
necessary, and set $\ell:=\frac{m-n}2$. We observe that for any
choice of $\Psi$ in Construction~\ref{psi} the subgroup
$C_{\Psi,\Sigma}$ lies in $G_\Sigma$ as an $\ell$-dimensional
complex subgroup. This follows from the fact that, since
$\Lambda_\C$ is the complexification of a real map, condition~(b)
of Construction~\ref{psi} is equivalent to $\Lambda_\C\Psi=0$,
which implies that $\exp\Lambda_\C(C_{\Psi,\Sigma})$ is trivial.

\begin{proposition}\label{toricfib}\
\begin{itemize}
\item[(a)]The toric variety $X_\Sigma$ is identified, as a
topological space, with the quotient of $\mathcal Z_{\sK_\Sigma}$
by the holomorphic action of $G_\Sigma/C_{\Psi,\Sigma}$.

\item[(b)]If the fan $\Sigma$ is regular, then $X_\Sigma$ is the base of a
holomorphic principal bundle with total space $\mathcal
Z_{\sK_\Sigma}$ and fibre the complex torus
$G_\Sigma/C_{\Psi,\Sigma}$ of dimension~$\ell$.
\end{itemize}
\end{proposition}
\begin{proof}
To prove~(a) we just observe that
\[
  X_\Sigma=U(\sK_\Sigma)/G_\Sigma=
  \bigl(U(\sK_\Sigma)/C_{\Psi,\Sigma}\bigr)\big/(G_\Sigma/C_{\Psi,\Sigma})\cong
  \mathcal Z_{\sK_\Sigma}\big/(G_\Sigma/C_{\Psi,\Sigma}),
\]
where we used Theorem~\ref{zkcomplex}. If $\Sigma$ is regular,
then $G_\Sigma\cong(\C^\times)^{m-n}$ and
$G_\Sigma/C_{\Psi,\Sigma}$ is a compact complex $\ell$-torus by
Example~\ref{2torus}. The action of $G_\Sigma$ on $U(\sK_\Sigma)$
is holomorphic and free in this case, and the same is true for the
action of $G_\Sigma/C_{\Psi,\Sigma}$ on $\mathcal Z_{\sK_\Sigma}$.
A holomorphic free actions of the torus $G_\Sigma/C_{\Psi,\Sigma}$
gives rise to a principal bundle, which finishes the proof of~(b).
\end{proof}

\begin{remark}
Like in the projective situation of~\cite{me-ve04}, for singular
varieties $X_\Sigma$ the quotient projection $\mathcal
Z_{\sK_\Sigma}\to X_\Sigma$ of Proposition~\ref{toricfib}~(a) is a
holomorphic principal \emph{Seifert bundle} for an appropriate
orbifold structure on~$X_\Sigma$. 
\end{remark}

Let $M$ be a complex $n$-dimensional manifold. The space
$\Omega_\C^*(M)$ of complex differential forms on $M$ decomposes
into a direct sum of the subspaces of \emph{$(p,q)$-forms},
$\Omega_\C^*(M)=\bigoplus_{0\le p,q\le n}\Omega^{p,q}(M)$, and
there is the \emph{Dolbeault differential}
$\bar\partial\colon\Omega^{p,q}(M)\to\Omega^{p,q+1}(M)$.
 The
dimensions $h^{p,q}(M)$ of the Dolbeault cohomology groups
$H_{\bar\partial}^{p,q}(M)$ are known as the \emph{Hodge numbers}
of~$M$. They are important invariants of the complex structure
of~$M$.


%
%

The Dolbeault cohomology of a compact complex $\ell$-dimensional
torus $T_\C^{2\ell}$ is isomorphic to an exterior algebra on
$2\ell$ generators:
\begin{equation}\label{dolbtorus}
  H_{\bar\partial}^{*,*}(T_\C^{2\ell})\cong
  \Lambda[\xi_1,\ldots,\xi_\ell,\eta_1,\ldots,\eta_\ell],
\end{equation}
where $\xi_1,\ldots,\xi_\ell\in
H_{\bar\partial}^{1,0}(T_\C^{2\ell})$ are the classes of basis
holomorphic 1-forms, and $\eta_1,\ldots,\eta_\ell\in
H_{\bar\partial}^{0,1}(T_\C^{2\ell})$ are the classes of basis
antiholomorphic 1-forms. In particular, the Hodge numbers are
given by $h^{p,q}(T_\C^{2\ell})=\binom\ell p\binom\ell q$.

The de Rham cohomology of a complete nonsingular toric variety
$X_\Sigma$ admits a Hodge decomposition with nonzero terms in
bidegrees $(p,p)$ only~{\cite[\S12]{dani78}}. This together with
the cohomology calculation due to
Danilov--Jurkiewicz~\cite[\S10]{dani78} gives the following
description of the Dolbeault cohomology:
\begin{equation}\label{dolbtoric}
  H_{\bar\partial}^{*,*}(X_\Sigma)\cong
  \C[v_1,\ldots,v_m]/(\mathcal I_{\sK_\Sigma}+\mathcal J_{\Sigma}),
\end{equation}
where $v_i\in H_{\bar\partial}^{1,1}(X_\Sigma)$, the ideal
$\mathcal I_{\sK_\Sigma}$ is generated by monomials $v_{i_1}\cdots
v_{i_k}$ for which $\mb a_{i_1},\ldots,\mb a_{i_k}$ do not span a
cone of $\Sigma$ (the \emph{Stanley--Reisner ideal}), and
$\mathcal J_{\Sigma}$ is generated by the linear combinations
$\sum_{k=1}^ma_{kj}v_k$ for $1\le j\le n$, where $a_{kj}$ denotes
the $j$th coordinate of~$\mb a_k$. We have
$h^{p,p}(X_\Sigma)=h_p$, where $(h_0,h_1,\ldots,h_n)$ is the
\emph{$h$-vector} of $\sK_\Sigma$~\cite[\S2.1]{bu-pa02}, and
$h^{p,q}(X_\Sigma)=0$ for $p\ne q$.

\setcounter{theorem}3

\begin{theorem}\label{dolbzp}
Let $\Sigma$ be a complete rational nonsingular fan in~$\R^n$,
with $m$ one-dimensional cones (some of which may be empty), and
$m+n$ even. Let $\zk=\mathcal Z_{\sK_\Sigma}$ be the moment-angle
manifold with the complex structure defined by a choice of
subgroup $C_{\Psi,\Sigma}$, see~\eqref{csigma}. Then the Dolbeault
cohomology algebra $H_{\bar\partial}^{*,*}(\zk)$ is isomorphic to
the cohomology of differential bigraded
algebra
\begin{equation}\label{zkmult}
\bigl[\Lambda[\xi_1,\ldots,\xi_\ell,\eta_1,\ldots,\eta_\ell]\otimes
  H_{\bar\partial}^{*,*}(X_\Sigma),d\bigr]
\end{equation}
whose bigrading is defined by~\eqref{dolbtorus}
and~\eqref{dolbtoric}, and differential $d$ of bidegree $(0,1)$ is
defined on the generators as
\[
  dv_i=d\eta_j=0,\quad d\xi_j=c(\xi_j),\quad
  1\le i\le m,\;1\le j\le\ell,
\]
where $c\colon H^{1,0}_{\bar\partial}(T_\C^{2\ell})\to
H^2(X_\Sigma, \C)=  H_{\bar\partial}^{1,1}(X_\Sigma)$ is the first
Chern class map of the torus principal bundle $\zk\to X_\Sigma$.
\end{theorem}

\begin{proof}
Let $(B,\delta)=\mathcal{M}(X_\Sigma)$ be the minimal Dolbeault
model~\cite{f-o-t08} of~$X_\Sigma$. Consider the differential
algebra
\begin{align*}
  &\bigl[\Lambda[\xi_1,\ldots,\xi_\ell,\eta_1,\ldots,\eta_\ell]\otimes
  B, d\bigr]\qquad\text{with}\\
  &d|_B=\delta,\quad  d(\xi_i)=c(\xi_i)\in
  B^{1,1}\cong H^{1,1}_{\bar\partial}(X_\Sigma),\quad d(\eta_i)=0.
\end{align*}
By \cite[Cor.~4.66]{f-o-t08}, it gives a model for the Dolbeault
cohomology ring of the total space $\zk$ of the principal
$T^{2\ell}_\C$-bundle $\zk\to X_\Sigma$, provided that $X_\Sigma$
is strictly formal in the sense of~\cite[Def.~4.58]{f-o-t08}.

According to~\cite[Th.~8]{ne-ta78}, a compact manifold satisfying
the $\partial\bar\partial$-lemma is strictly formal, in particular
any complete toric variety is strictly formal. The ordinary
formality implies the existence of a quasi-isomorphism
$\phi_B\colon B\to H_{\bar\partial}^{*,*}(X_\Sigma)$, which
extends to a quasi-isomorphism
\[
  \mbox{id}\otimes\phi_B\colon
  \bigl[\Lambda[\xi_1,\ldots,\xi_\ell,\eta_1,\ldots,\eta_\ell]
  \otimes B, d\bigr]\to
  \bigl[\Lambda[\xi_1,\ldots,\xi_\ell,\eta_1,\ldots,\eta_\ell]
  \otimes H_{\bar\partial}^{*,*}(X_\Sigma), d\bigr]
\]
by~\cite[Lemma~14.2]{f-h-t01}. Thus, the differential algebra
$\bigl[\Lambda[\xi_1,\ldots,\xi_\ell,\eta_1,\ldots,\eta_\ell]\otimes
H^{*,*}_{\bar\partial}(X_\Sigma), d\bigr]$ provides a model for
the Dolbeault cohomology of $\zk$, as claimed.
\end{proof}

The map $c$ can be described explicitly in terms of the matrix
$\Psi$ defining the complex structure on $\zk$. We shall also need
the real $(m-n)\times m$ matrix $\varGamma=(\gamma_{jk})$ such
that $(\gamma_{j1},\ldots,\gamma_{jm})$, $1\le j\le m-n$, is any
basis in the space of linear relations between the generators $\mb
a_1,\ldots,\mb a_m$ of the one-dimensional cones of~$\Sigma$.
(This matrix was already considered in Section~4 for the polytopal
case.)

\begin{lemma}\label{mumatrix}
We have
\[
  c(\xi_j)=\mu_{j1}v_1+\ldots+\mu_{jm}v_m,\quad 1\le j\le \ell,
\]
where $M=(\mu_{jk})$ is an $\ell\times m$ matrix satisfying the
two conditions:
\begin{itemize}
\item[(a)] $\varGamma M^t\colon\C^\ell\to\C^{2\ell}$ is a
monomorphism;

\item[(b)] $M\Psi=0$.
\end{itemize}
\end{lemma}
\begin{proof}
Let $\Lambda^t_\C\colon\C^n\to\C^m$ be the transpose of the
complexified map~\eqref{lambdar}. Then have
$H^1(T_\C^{2\ell};\C)=\C^m/(\Lambda^t_\C(\C^n))$ and
$H^2(X_\Sigma;\C)=\C^{m-k}/(\Lambda^t_\C(\C^n))$, where $k$ is the
number of ghost vertices in~$\sK$. Map $c$ is given by
the projection
\[
  p\colon \C^m\big/
  \bigl(\Lambda^t_\C(\C^n)\bigr)
  \longrightarrow \C^{m-k}/\bigl(\Lambda^t_\C(\C^n)\bigr)
\]
which forgets the coordinates in $\C^m$ corresponding to the ghost
vertices. We therefore need to identify the subspace of
holomorphic differentials
$H^{1,0}_{\bar\partial}(T_\C^{2\ell})\cong\C^\ell$ inside the
space of all 1-forms $H^1(T_\C^{2\ell};\C)\cong\C^{2\ell}$.
Since
\[
  T_\C^{2\ell}=G_\Sigma/C_{\Psi,\Sigma}=(\Ker\exp\Lambda_\C)/\exp\Psi(\C^\ell),
\]
holomorphic differentials on $T_\C^{2\ell}$ correspond to linear
functions on $\Ker\Lambda_\C$ which are $\Psi(\C^\ell)$-invariant.
Every such linear function is a restriction of a
$\Psi(\C^\ell)$-invariant linear function on $\C^m$ to
$\Ker\Lambda_\C\subset\C^m$. The kernel of this restriction
consists of those functions which are
$\Lambda^t_\C(\C^n)$-invariant. Condition~(b) says exactly that
the rows of $M$ are $\Psi(\C^\ell)$-invariant linear functions.
Condition~(a) says that the rows of $M$ constitute a basis in the
complement to the subspace of $\Lambda^t_\C(\C^n)$-invariant
functions in the space of $\Psi(\C^\ell)$-invariant functions.
\end{proof}


It is interesting to compare Theorem~\ref{dolbzp} with the
following description of the de Rham cohomology of~$\zk$.

\setcounter{theorem}6

\begin{theorem}[{\cite[Th.~7.36]{bu-pa02}}]\label{cohomzpred}
Let $\zk$ be as in Theorem~\ref{dolbzp}. Then its de Rham
cohomology algebra is isomorphic to the cohomology of the
differential graded algebra
\[
  \bigl[\Lambda[u_1,\ldots,u_{m-n}]\otimes
  H^*(X_\Sigma),d\bigr],
\]
with $\deg u_j=1$, $\deg v_i=2$, and differential $d$ defined on
the generators as
\[
  dv_i=0,\quad du_j=\gamma_{j1}v_1+\ldots+\gamma_{jm}v_m,\quad
  1\le j\le m-n.
\]
\end{theorem}

This follows from the more general result~\cite[Th.~7.7]{bu-pa02}
describing the cohomology of~$\zk$. Like Theorem~\ref{dolbzp}, the
theorem above may be interpreted as a collapse result for a
spectral sequence, this time the Leray--Serre spectral sequence of
the principal $T^{m-n}$-bundle $\zk\to X_\Sigma$. For more
information about the ordinary cohomology of $\zk$
see~\cite[Ch.~7]{bu-pa02} and~\cite[\S4]{pano08}. There is also a
bigrading in the ordinary cohomology of $\zk$, which is different
from the bigrading in the Dolbeault cohomology. These two
bigradings may be merged in the Dolbeault cohomology, providing a
four-graded structure.

The are two classical spectral sequences for the Dolbeault
cohomology. First, the \emph{Borel spectral sequence}~\cite{bore}
of a holomorphic fibre bundle $E\to B$ with a compact K\"ahler
fibre~$F$, whose $E_2$ term is $H_{\bar\partial}(B)\otimes
H_{\bar\partial}(F)$ and which converges to $H_{\bar\partial}(E)$.
Second, the \emph{Fr\"olicher spectral
sequence}~\cite[\S3.5]{gr-ha78}, whose $E_1$ term is the Dolbeault
cohomology of a complex manifold $M$ and which converges to the de
Rham cohomology of~$M$. Theorems~\ref{dolbzp} and~\ref{cohomzpred}
can be interpreted as collapse results for these spectral
sequences:

\begin{corollary}\
\begin{itemize}
\item[(a)]
The Borel spectral sequence of the holomorphic principal bundle
$\zk\to X_\Sigma$ collapses at the $E_3$ term, i.e.
$E_3=E_\infty$;

\item[(b)]
the Fr\"olicher spectral sequence of $\zk$ collapses at the $E_2$
term: $E_2=E_\infty$.
\end{itemize}
\end{corollary}
\begin{proof}
To prove~(a) we just observe that the differential algebra given
by~\eqref{zkmult} is the $E_2$ term of the Dolbeault spectral
sequence, and its cohomology is the $E_3$ term.

Let us prove~(b). By comparing the Dolbeault and de Rham
cohomology algebras of $\zk$ given by Theorems~\ref{dolbzp}
and~\ref{cohomzpred} we observe that the elements
$\eta_1,\ldots,\eta_\ell\in E_1^{0,1}$ cannot survive in
the~$E_\infty$ term of the Fr\"olicher spectral sequence. The only
possible nontrivial differential on these elements is $d_1\colon
E_1^{0,1}\to E_1^{1,1}$. By Theorem~\ref{cohomzpred}, the
cohomology algebra of $[E_1,d_1]$ is exactly the de Rham
cohomology of~$\zk$.
\end{proof}

\begin{lemma}\label{cbounds}
Let $k$ be the number of ghost vertices in~$\sK$.
Then
\[
  k-\ell\le\dim_\C\Ker\bigl(c\colon H^{1,0}_{\bar\partial}(T_\C^{2\ell})\to
  H_{\bar\partial}^{1,1}(X_\Sigma)\bigr)\le\sbr k2.
\]
In particular, if $k\le1$ then the map $c$ is monic.
\end{lemma}
\begin{proof}
The map $c$ is given by the composite map in the top line of the
diagram
\[
\begin{CD}
  \C^\ell @>\Theta>> \C^m\big/\bigl(\Lambda^t_\C(\C^n)\bigr)
  @>p>> \C^{m-k}\big/\bigl(\Lambda^t_\C(\C^n)\bigr)\\
  @VV{\cong}V @VV\mathrm{Re}V @VV\mathrm{Re}V\\
  \R^{m-n}@=\R^{m-n} @>p'>> \R^{m-n-k}
\end{CD},
\]
where $\Theta$ is an inclusion of a complex subspace described in
the proof of Lemma~\ref{mumatrix} and $p$, $p'$ are the
projections forgetting the ghost vertices. The left vertical arrow
is an ($\R$-linear) isomorphism, as it has the form
$H^{1,0}_{\bar\partial}(T_\C^{2\ell})\to H^1(T_\C^{2\ell},\C)\to
H^1(T_\C^{2\ell},\R)$, and any real-valued function on the lattice
determining the torus $T_\C^{2\ell}$ is the real part of the
restriction of a $\C$-linear function to the same lattice.

Since the diagram above is commutative, the kernel of
$c=p\circ\Theta$ has the real dimension at most~$k$, which implies
the required upper bound on its complex dimension. For the lower
bound, we have $\dim_\C\Ker c\ge\dim_\C
H^{1,0}_{\bar\partial}(T_\C^{2\ell})-\dim_\C
H_{\bar\partial}^{1,1}(X_\Sigma) =\ell-(2\ell-k)=k-\ell$.
\end{proof}

%
\begin{theorem}\label{hodge}
Let $\zk$ be as in Theorem~\ref{dolbzp}, and let $k$ be the number
of ghost vertices in~$\sK$. Then the Hodge numbers
$h^{p,q}=h^{p,q}(\zk)$ satisfy
\begin{itemize}
\item[(a)] $\binom{k-\ell}p \le h^{p,0}\le\binom {[k/2]}p$ for $p\ge0$;
\item[(b)] $h^{0,q}=\binom\ell q$ for $q\ge0$;
\item[(c)] $h^{1,q}
  =(\ell-k)\binom\ell{q-1}+h^{1,0}\binom{\ell+1}q$ for $q\ge1$;
\item[(d)] $\frac{\ell(3\ell+1)}2-h_2(\sK)-\ell k+(\ell+1)h^{2,0}
 \le h^{2,1}\le\frac{\ell(3\ell+1)}2-\ell k+(\ell+1)h^{2,0}$.
\end{itemize}
\end{theorem}
\begin{proof}
Let $A^{p,q}$ denote the bidegree $(p,q)$ component of the
differential algebra from Theorem~\ref{dolbzp}, and let
$Z^{p,q}\subset A^{p,q}$ denote the subspace of $d$-cocycles. Then
$d^{1,0}\colon A^{1,0}\to Z^{1,1}$ coincides with the map~$c$, and
the required bounds for $h^{1,0}=\Ker d^{1,0}$ are already
established in Lemma~\ref{cbounds}. Since $h^{p,0}=\Ker d^{p,0}$
is the $p$th exterior power of the space $\Ker d^{1,0}$,
statement~(a) follows.

The differential is trivial on $A^{0,q}$, hence $h^{0,q}=\dim
A^{0,q}$, proving~(b).

The space $Z^{1,1}$ is spanned by the $v_i$ and $\xi_i\eta_j$ with
$\xi_i\in\Ker d^{1,0}$. Hence, $\dim Z^{1,1}=2\ell-k+h^{1,0}\ell$.
Also, $\dim d(A^{1,0})=\ell-h^{1,0}$, hence
$h^{1,1}=\ell-k+h^{1,0}(\ell+1)$. Similarly, $\dim
Z^{1,q}=(2\ell-k)\binom\ell{q-1}+h^{1,0}\binom\ell q$ (spanned by
the elements $v_i\eta_{j_1}\cdots\eta_{j_{q-1}}$ and
$\xi_i\eta_{j_1}\cdots\eta_{j_q}$ where $\xi_i\in\Ker d^{1,0}$,
$j_1<\ldots<j_q$), and $d\colon A^{1,q-1}\to Z^{1,q}$ hits a
subspace of dimension $(\ell-h^{1,0})\binom\ell{q-1}$. This
proves~(c).

We have $A^{2,1}=V\oplus W$, where $V$ is spanned by the monomials
$\xi_iv_j$ and $W$ by the monomials $\xi_i\xi_j\eta_k$.
Therefore,
\begin{equation}\label{h21}
  h^{2,1}=\dim V-\dim dV+\dim W-\dim dW-\dim dA^{2,0}.
\end{equation}
Now $\dim V=\ell(2\ell-k)$, $0\le\dim dV\le h_2(\sK)$ (since
$dV\subset H_{\bar\partial}^{2,2}(X_\Sigma))$, $\dim W-\dim
dW=\dim\Ker d|_W=\ell h^{2,0}$, and $\dim
dA^{2,0}=\binom\ell2-h^{2,0}$. Plugging these values
into~\eqref{h21} we obtain the inequalities of~(d).
\end{proof}

\begin{remark}
At most one ghost vertex is required to make $\dim\zk=m+n$ even.
Note that $k\le1$ implies $h^{p,0}(\zk)=0$, so that $\zk$ does not
have holomorphic forms of any degree in this case. If $\zk$ is a
torus, then $m=k=2\ell$, and $h^{1,0}(\zk)=h^{0,1}(\zk)=\ell$.
Otherwise Theorem~\ref{hodge} implies that
$h^{1,0}(\zk)<h^{0,1}(\zk)$, and therefore $\zk$ is not K\"ahler
(in the polytopal case this was observed in~\cite[Th.~3]{meer00}).
\end{remark}

\begin{example}Let $\zp\cong S^1\times S^{2n+1}$ be a Hopf manifold of
Example~\ref{hopf}. The corresponding fan is the normal fan
$\Sigma_P$ of the standard $n$-dimensional simplex $P$ with one
redundant inequality. We have $X_P=\C P^n$, and~\eqref{dolbtoric}
describes its cohomology as the quotient of
$\C[v_1,\ldots,v_{n+2}]$ by the two ideals: $\mathcal I$ generated
by $v_1$ and $v_2\cdots v_{n+2}$, and $\mathcal J$ generated by
$v_2-v_{n+2},\ldots,v_{n+1}-v_{n+2}$. The differential algebra of
Theorem~\ref{dolbzp} is therefore given by
$\bigl[\Lambda[\xi,\eta]\otimes\C[t]/t^{n+1},d\bigr]$, 
and it is straightforward to check that $dt=d\eta=0$ and $d\xi=t$
for a proper choice of~$t$. The nontrivial cohomology classes are
represented by the cocycles $1$, $\eta$, $\xi t^n$ and $\xi\eta
t^n$, which gives the following nonzero Hodge numbers of~$\zp$:
$h^{0,0}=h^{0,1}=h^{n+1,n}=h^{n+1,n+1}=1$.
\end{example}

\begin{example}[Calabi--Eckmann manifold]
Let $P=\Delta^p\times\Delta^q$ be the product of two standard
simplices with $p\le q$, so that $n=p+q$, $m=n+2$ and $\ell=1$.
The corresponding toric variety is $X_P=\C P^p\times \C P^q$ and
its cohomology ring is isomorphic to $\C[x,y]/(x^{p+1}, y^{q+1})$.
We may choose $\Psi=(1,\ldots,1,\alpha,\ldots,\alpha)^t$ in
Construction~\ref{psi}, where the number of units is $p+1$, the
number of $\alpha$'s is $q+1$ and $\alpha\notin\R$. This provides
$\zp\cong S^{2p+1}\times S^{2q+1}$ with a structure of a complex
manifold of dimension~$p+q+1$, known as a \emph{Calabi--Eckmann
manifold}; we denote complex manifolds obtained in this way by
$\mbox{\textit{C\!E}}(p,q)$. They are total spaces of holomorphic
principal bundles $\mbox{\textit{C\!E}}(p,q)\to \C P^p\times \C
P^q$ with fibre the complex torus~$\C/(\Z\oplus\alpha\Z)$.

Theorem~\ref{dolbzp} and Lemma~\ref{mumatrix} provide the
following description of the Dolbeault cohomology of
$\mbox{\textit{C\!E}}(p,q)$:
\[
  H^{*,*}_{\bar\partial}\bigl(\mbox{\textit{C\!E}}(p,q)\bigr)\cong
  H\bigl[\Lambda[\xi,\eta]\otimes\C[x,y]/(x^{p+1},y^{q+1}),d\bigr],
\]
where $dx=dy=d\eta=0$ and $d\xi=x-y$ for an appropriate choice of
$x,y$. We therefore obtain
\begin{equation}\label{dolbce}
  H^{*,*}_{\bar\partial}\bigl(\mbox{\textit{C\!E}}(p,q)\bigr)\cong
  \Lambda[\omega,\eta]\otimes\C[x]/(x^{p+1}),
\end{equation}
where $\omega\in
H^{q+1,q}_{\bar\partial}\bigl(\mbox{\textit{C\!E}}(p,q)\bigr)$ is
the cohomology class of the cocycle
$\xi\frac{x^{q+1}-y^{q+1}}{x-y}$. This calculation was done
in~\cite[\S9]{bore} using a slightly different argument. We note
that Dolbeault cohomology of a Calabi--Eckmann manifold depends
only on $p,q$ and does not depend on the complex
parameter~$\alpha$ (or matrix~$\Psi$).

\begin{example}
Now let $P=\Delta^1\times\Delta^1\times\Delta^2\times\Delta^2$.
Then $\zp$ has two structures of a product of Calabi--Eckmann
manifolds, namely,
$\mbox{\textit{C\!E}}(1,1)\times\mbox{\textit{C\!E}}(2,2)$ and
$\mbox{\textit{C\!E}}(1,2)\times\mbox{\textit{C\!E}}(1,2)$. Using
isomorphism~\eqref{dolbce} we observe that these two complex
manifolds have different Hodge numbers~$h^{2,1}$: it is 1 in the
first case, and 0 in the second. This shows that the choice of
matrix $\Psi$ affects not only the complex structure of~$\zp$, but
also its Hodge numbers, unlike the previous examples of complex
tori, Hopf and Calabi--Eckmann manifolds. Certainly it is not
highly surprising from the complex-analytic point of view.
\end{example}
\end{example}

\end{document}